\documentclass[A4,12pt,reqno]{amsart}

\usepackage{amssymb,amsthm}
\usepackage{amsfonts,cmtiup,comment,stmaryrd}

\usepackage{mathrsfs,cmtiup}

\makeatletter
\def\blfootnote{\xdef\@thefnmark{}\@footnotetext}
\makeatother

\newtheorem{theorem}{Theorem}[section]
\newtheorem{lemma}[theorem]{Lemma}
\newtheorem{proposition}[theorem]{Proposition}
\newtheorem{corollary}[theorem]{Corollary}

\theoremstyle{definition}

\newtheorem{remark}[theorem]{Remark}
\newtheorem*{definition*}{Definition}

\newcommand{\N}{\Bbb N}

\newcommand{\s}{\sigma }

\newcommand{\f}{\varphi}

\newcommand{\g}{\gamma }

\renewcommand{\geq}{\geqslant}
\renewcommand{\leq}{\leqslant}

\newcommand{\ed} {\end{document}}

\let\leq=\leqslant
\let\geq=\geqslant
\numberwithin{equation}{section}

\begin{document}
\title{Compact groups in which all elements have countable right Engel sinks}

\author{E. I. Khukhro}
\address{Charlotte Scott Research Centre for Algebra, University of Lincoln, U.K., and \newline \indent  Sobolev Institute of Mathematics, Novosibirsk, 630090, Russia}
\email{khukhro@yahoo.co.uk}

\author{P. Shumyatsky}

\address{Department of Mathematics, University of Brasilia, DF~70910-900, Brazil}
\email{pavel@unb.br}


\begin{abstract}
A right Engel sink of an element $g$ of a group $G$ is a set ${\mathscr R}(g)$ such that for every $x\in G$ all sufficiently long commutators $[...[[g,x],x],\dots ,x]$ belong to ${\mathscr R}(g)$.  (Thus, $g$ is a right Engel element precisely when we can choose ${\mathscr R}(g)=\{ 1\}$.) It is proved that if every element of a compact (Hausdorff) group $G$ has a countable (or finite) right Engel sink, then $G$ has a finite normal subgroup $N$ such that $G/N$ is locally nilpotent.
\end{abstract}
\maketitle
\section{Introduction}

A group $G$ is called an Engel group if for every $x,g\in G$ the equation $[x,\,{}_{n} g]=1$ holds for some $n=n(x,g)$ depending on $x$ and $g$.
Henceforth, we use the left-normed simple commutator notation
$[a_1,a_2,a_3,\dots ,a_r]:=[...[[a_1,a_2],a_3],\dots ,a_r]$ and the abbreviation $[a,\,{}_kb]:=[a,b,b,\dots, b]$ where $b$ is repeated $k$ times. A group is said to be locally nilpotent if every finite subset generates a nilpotent subgroup. Clearly, any locally nilpotent group is an Engel group. Wilson and Zelmanov \cite{wi-ze} proved the converse for profinite groups: any Engel profinite group is locally nilpotent. Later Medvedev \cite{med} extended this result to Engel compact  groups. (Henceforth by compact groups we mean compact Hausdorff groups.)

Generalizations of Engel groups can be defined in terms of Engel sinks.

\begin{definition*} \label{d}
 A \textit{left Engel sink} of an element $g$ of a group $G$ is a set ${\mathscr E}(g)$ such that for every $x\in G$ all sufficiently long commutators $[x,g,g,\dots ,g]$ belong to ${\mathscr E}(g)$, that is, for every $x\in G$ there is a positive integer $l=l(x,g)$ such that
 $[x,\,{}_{l}g]\in {\mathscr E}(g)$ for all $l\geq l(x,g).
 $

  A \textit{right Engel sink} of an element $g$ of a group $G$ is a set ${\mathscr R}(g)$ such that for every $x\in G$ all sufficiently long commutators $[g,x,x,\dots ,x]$ belong to ${\mathscr R}(g)$, that is, for every $x\in G$ there is a positive integer $r(x,g)$ such that
 $[g,\,{}_{r}x]\in {\mathscr R}(g)$ for all $r\geq r(x,g).
 $
 \end{definition*}
 \noindent (Thus, $g$ is a left Engel element precisely when we can choose ${\mathscr E}(g)=\{ 1\}$, and a right Engel element when we can choose ${\mathscr R}(g)=\{ 1\}$.)

Earlier we considered in \cite{khu-shu,khu-shu191} compact groups $G$ in which  every element has a countable or finite left Engel sink and proved the following theorem.

\begin{theorem}[{\cite[Theorem~1.2]{khu-shu191}}]\label{t-left}
If every element of a compact group $G$ has a countable
left Engel sink, then $G$ has a finite normal subgroup $N$ such that $G/N$ is locally nilpotent.
\end{theorem}

(Henceforth by ``countable'' we mean ``finite or denumerable''.)

For right Engel sinks we proved earlier in \cite{khu-shu172} that if every element of a compact group has a finite right Engel sink, then the group is finite-by-(locally nilpotent). In the present paper we extend this result to countable right Engel sinks.

\begin{theorem}\label{t-main}
Suppose that $G$ is a compact group in which every element has a countable right Engel sink. Then $G$ has a finite normal subgroup $N$ such that $G/N$ is locally nilpotent.
\end{theorem}

 In Theorem~\ref{t-main} it also follows that there is a locally nilpotent subgroup of finite index --- just consider $C_G( N)$.

While it is well-known that the inverse of a right Engel element is a left Engel element,
it is unclear if the existence of a countable (or finite) right Engel sink of a given element implies the existence of a countable (or finite) left Engel sink of this element or its inverse. It is only by virtue of our Theorem~\ref{t-main} that if all elements of a compact group have countable right Engel sinks, then in fact all elements have finite right and left Engel sinks contained in the same finite normal subgroup.

The proof uses the aforementioned Wilson--Zelmanov theorem for profinite groups. First the case of pro-$p$ groups  is considered, where Lie ring methods are applied including Zelmanov's theorem on Lie algebras satisfying a polynomial identity and generated by elements all of whose products are ad-nilpotent
\cite{ze92,ze95,ze17}. As we noted in \cite{khu-shu172}, it is easy to see that if every element  of a pro-$p$ group has a finite right Engel sink,  then the group is locally
nilpotent. But in the present paper, with countable right Engel  sinks, the case of pro-$p$ groups requires substantial efforts.
Then the case of prosoluble groups is settled by using properties of coprime actions including a profinite analogue of a theorem of Thompson~\cite{th}.
The general case of profinite groups is dealt with by bounding the nonsoluble length of the group, which enables induction on this length. (We introduced the nonsoluble length in \cite{khu-shu131}, although bounds for nonsoluble length had been implicitly used in various earlier papers, for example, in the celebrated Hall--Higman paper \cite{ha-hi},
or in Wilson's  paper \cite{wil83}; more recently, bounds for the nonsoluble length were used in the study of verbal subgroups in finite and profinite groups \cite{dms1, 68, austral, khu-shu132}.)
Finally, the result for compact groups is derived with the use of the structure theorems for compact groups.

\section{Preliminaries}
In this section we recall some  notation and terminology and establish some general properties of left and right Engel sinks in compact and profinite groups.

Our notation and terminology for profinite and compact groups is standard; see, for example,  \cite{rib-zal},  \cite{wil}, and \cite{hof-mor}.  A subgroup (topologically) generated by a subset $S$ is denoted by $\langle S\rangle$. Recall that centralizers are closed subgroups, while commutator subgroups $[B,A]=\langle [b,a]\mid b\in B,\;a\in A\rangle$ are the closures of the corresponding abstract commutator subgroups.

For a group $A$ acting by automorphisms on a group $B$ we use the usual notation for commutators $[b,a]=b^{-1}b^a$ and commutator subgroups $[B,A]=\langle [b,a]\mid b\in B,\;a\in A\rangle$, as well as for centralizers $C_B(A)=\{b\in B\mid b^a=b \text{ for all }a\in A\}$
and $C_A(B)=\{a\in A\mid b^a=b\text{ for all }b\in B\}$.

We record for convenience the following simple lemma.

\begin{lemma}[{\cite[Lemma~2.1]{khu-shu191}}] \label{l-fng}
Suppose that $\varphi$ is a continuous automorphism of a compact group $G$ such that
$G=[G,\varphi ]$. If $N$ is a normal subgroup of $G$ contained in $C_G(\varphi )$, then $N\leq Z(G)$.
\end{lemma}

We denote by $\pi (k)$ the set of prime divisors of $k$, where $k$ may be a positive integer or a Steinitz number, and by $\pi (G)$ the set of prime divisors of the orders of elements of a (profinite) group $G$. Let $\sigma$ be a set of primes. An element $g$ of a group is  a $\sigma$-element if $\pi(|g|)\subseteq \sigma$, and a group $G$ is a $\sigma$-group if all of its elements are $\sigma$-elements. We denote by $\sigma'$ the complement of $\sigma$ in the set of all primes. When $\sigma=\{p\}$,  we write $p$-element, $p'$-element, etc.

Recall that a pro-$p$ group  is an inverse limit of finite $p$-groups, a pro-$\sigma $ group is an inverse limit of finite $\sigma$-groups, a pronilpotent group is an inverse limit of finite nilpotent groups, a prosoluble group is an inverse limit of finite soluble groups.

We denote by  $\gamma _{\infty}(G)=\bigcap _i\gamma _i(G)$ the intersection of the lower central series of a group $G$. A profinite group $G$ is pronilpotent if and only if $\gamma _{\infty}(G)=1$.

Profinite groups have Sylow $p$-subgroups and satisfy analogues of the Sylow theorems.  Prosoluble groups satisfy analogues of the
theorems 
 on Hall $\pi$-subgroups. 
 We refer the reader to the corresponding chapters in \cite[Ch.~2]{rib-zal} and \cite[Ch.~2]{wil}.
 We add a simple folklore lemma.

\begin{lemma}[{\cite[Lemma~2.2]{khu-shu191}}]\label{l-prosol-by-prosol}
A profinite group $G$ that is an extension of a prosoluble group $N$ by a prosoluble group $G/N$ is prosoluble.
\end{lemma}

We shall use several times the following well-known fact, which is straightforward from the Baire Category Theorem (see \cite[Theorem~34]{kel}).

\begin{theorem}\label{bct}
If a compact Hausdorff group is a countable union of closed subsets, then one of these subsets has non-empty interior.
\end{theorem}

We now establish some general properties of Engel sinks.
Clearly, the intersection of two left Engel sinks of a given element $g$ of a group $G$ is again a left Engel sink of $g$, with the corresponding function $l(x,g)$ being the maximum of the two functions. Therefore, if $g$ has a \textit{finite} left Engel sink, then $g$ has a unique smallest left  Engel sink, which  has the following  characterization.

\begin{lemma}[{\cite[Lemma~2.1]{khu-shu}}]\label{l-min} If an element $g$ of a group $G$ has a finite left Engel sink, then $g$ has a smallest left Engel sink $\mathscr E (g)$
and for every $s\in \mathscr E (g)$ there is an integer $k\geq 1$ such that  $s=[s,\,{}_kg]$.
\end{lemma}

There are similar observations about right Engel sinks. The intersection of two right Engel sinks of a given element $g$ of a group $G$ is again a right Engel sink of $g$, with the corresponding function $r(x,g)$ being the maximum of the two functions. Therefore, if $g$ has a \textit{finite} right Engel sink, then $g$ has a unique smallest right Engel sink, which  has the following  characterization.

\begin{lemma}[{\cite[Lemma~2.2]{khu-shu172}}]\label{l-min-r}
If an element $g$ of a group $G$ has a finite right Engel sink, then $g$ has a smallest right Engel sink $\mathscr R(g)$ and for every $z\in \mathscr R(g)$ there are integers $n\geq 1$ and $m\geq 1$ and an element $x\in G$ such that $z=[g,{}_nx]=[g,{}_{n+m}x]$.
\end{lemma}

(Of course, the elements $x$ and numbers $m,n$ in the above lemma vary for different $z$ and are not unique.)

The following lemma was proved by Heineken \cite{hei}.

\begin{lemma}[{\cite[12.3.1]{rob}}]\label{l-hei}
 If $g$ is a right Engel element of a group $G$, then $g^{-1}$ is a left Engel element.
\end{lemma}

Furthermore, for metabelian groups we have the following.

\begin{lemma}[{\cite[Lemma~2.5]{khu-shu172}}]\label{l-metab}
 If $G$ is a metabelian group, then a right Engel sink of the inverse $g^{-1}$ of an element $g\in G$ is a left Engel sink of $g$.
\end{lemma}

Thus, if $G$ is a metabelian group in which all elements have finite right Engel sinks, then all elements of $G$ also have finite left Engel sinks, and if all elements of $G$ have countable right Engel sinks, then all elements of $G$ also have countable left Engel sinks.

\begin{remark} If every element of a group has a countable right Engel sink, then this condition is inherited by every section of the group, and we shall use this property without special references. The same applies to a group in which every element has a finite right Engel sink. Similar properties hold for left Engel sinks. \end{remark}

\section{Pronilpotent groups}

When $G$ is a pro-$p$ group, or more generally a pronilpotent group, the conclusion of the main Theorem~\ref{t-main} is equivalent to $G$ being locally nilpotent, and this is what we prove in this section.

\begin{theorem}
\label{t2}
Suppose that $G$ is a pronilpotent group in which every element has a countable right Engel sink. Then $G$ is locally nilpotent.
\end{theorem}

First we establish an Engel-like property.

\begin{lemma}\label{l-ksm}
Suppose that $G$ is a profinite group in which every element has a countable right Engel sink. For any elements $a,b\in G$ there exist positive integers $k,s,m$ (depending on $a,b$) such that
\begin{equation*}
 [[b,\,{}_ka^{s}],a^{t}]=1.
\end{equation*}

\end{lemma}
\begin{proof}
Let    $\{s_1,s_2,\dots \}$ be a countable right Engel sink of $b$.
Consider the subsets
$$
T_{i,j,k}=\{x\in \langle a\rangle \mid [b,\,{}_kx]=s_i\text{ and }  [s_i,x]=s_j\}
$$
(where $\langle a\rangle$ is the procyclic subgroup generated by $a$). Note that each $T_{i,j,k}$ is a closed subset of $\langle a\rangle$.

By the definition of a right Engel sink, we have
$$
\langle a\rangle=\bigcup _{i,j,k}T_{i,j,k}.
$$
By Theorem~\ref{bct} some $T_{i,j,k}$ contains an open subset of $\langle a\rangle$, so $Nd\subset T_{i,j,k}$ for some open subgroup $N$ of $\langle a\rangle$ and some $d\in \langle a\rangle$. Since $\langle a\rangle/N$ is finite, we can assume that $d=a^{s}$ for some positive integer $s$.

Since $[s_i,nd]=s_j$ for all $n\in N$, it follows that $[s_i,N]=1$. Since $\langle a\rangle/N$ is finite, we have $a^{t}\in N$ for some positive integer $t$, so that $[s_i,a^{t}]=1$. As a result,
 \begin{align*}
 [[b,\,{}_ka^{s}],a^{t}]= [[b,\,{}_kd],a^{t}]
 &=[s_i,a^{t}]=1.\qedhere
 \end{align*}
\end{proof}

The bulk of the proof of Theorem~\ref{t2} is about the case where $G$ is a pro-$p$ group. First we
remind the reader of important Lie ring methods in the theory of pro-$p$ groups.

For a prime number $p $, the \textit{Zassenhaus $p $-filtration} of a group $G$ (also called the \textit{$p $-dimension series}) is defined by
$$
G_i=\langle g^{p ^k}\mid g\in \gamma _j(G),\;\, jp ^k\geqslant i\rangle \qquad\text{for}\quad i\in \N .
$$
This is indeed a \textit{filtration} (or an \textit{$N$-series}, or a \textit{strongly central series}) in the sense that
\begin{equation}\label{e-fil}
[G_i,G_j] \leqslant G_{i+j}\qquad \text{for all}\quad i, j.
 \end{equation}

 Then the Lie ring $D_p (G)$ is defined with the additive group
$$
D_p(G)=\bigoplus _{i}G_i/G_{i+1},
$$
where the  factors $Q_i=G_i/G_{i+1}$ are additively written. The Lie product is defined on homogeneous elements $xG_{i+1}\in Q_i$, $yG_{j+1}\in Q_j$ via the group commutators by
$$
[xG_{i+1},\, yG_{j+1}] = [x, y]G_{i+j+1}\in Q_{i+j}
$$
and extended to arbitrary elements of $D_p(G)$ by linearity. Condition~\eqref{e-fil} ensures that this  product is well-defined, and group commutator identities imply that $D_p(G)$ with these operations is a Lie ring. Since all the factors $G_i/G_{i+1}$ have prime exponent~$p $, we can view $D_p(G)$ as a Lie algebra over the field of $p $ elements $\mathbb{F}_p $. We denote  by $L_p (G)$ the subalgebra generated by the first factor $G/G_2$. (Sometimes, the notation $L_p (G)$ is used for $D_p (G)$.) If $u\in G_i\setminus G_{i+1}$, then we define $\delta  (u)=i$ to be the \textit{degree} of  $u$ with respect to the Zassenhaus filtration.

A group $G$ is said to satisfy a \textit{coset identity}  if there is a group word $w(x_1,\dots ,x_m)$,
elements $a_1,\dots,a_m$, and a subgroup $H \leq G$ such that $w(a_1h_1,\dots,a_mh_m) = 1$ for any $h_1,\dots,h_m\in H$.
 We shall use the following  result of Wilson and Zelmanov \cite{wi-ze} about coset identities.

\begin{theorem}[{Wilson and Zelmanov \cite[Theorem~1]{wi-ze}}]\label{t-coset}
If a group $G$ satisfies a coset identity on cosets of a subgroup of finite index, then for every prime $p$ the Lie algebra $L_p (G)$ constructed with respect to the Zassenhaus $p $-filtration satisfies a polynomial identity.
\end{theorem}

 Theorem~\ref{t-coset} was used in the proof of the above-mentioned theorem on profinite Engel groups, which we state here for convenience.

\begin{theorem}[{Wilson and Zelmanov \cite[Theorem~5]{wi-ze}}]\label{t-wz}
Every profinite Engel group is locally nilpotent.
\end{theorem}

The proof of Theorem~\ref{t-wz} was based on the following deep result of Zelmanov \cite{ze92,ze95,ze17}, which is also used in our paper.

\begin{theorem}[{Zelmanov \cite{ze92,ze95,ze17}}]\label{tz}
 Let $L$ be a Lie algebra over a field and suppose that $L$ satisfies
a polynomial identity. If $L$ can be generated by a finite set $X$ such that every
commutator in elements of $X$ is ad-nilpotent, then $L$ is nilpotent.
\end{theorem}

We now consider pro-$p$ groups with countable right Engel sinks.

\begin{proposition}
\label{pr-pro-p}
Suppose that $P$ is a finitely generated pro-$p$ group in which every element  has a countable right Engel sink. Then $P$ is nilpotent.
\end{proposition}

\begin{proof}
Our immediate aim is an application of
Theorem~\ref{tz} to the Lie algebra $L_p(P)$ of $P$, which will show that $L_p(P)$ is nilpotent. We need to verify that the conditions of that theorem are satisfied.

\begin{lemma}\label{l-ad}
The Lie algebra $L_p(P)$ is generated by finitely many elements all commutators in which are ad-nilpotent.
\end{lemma}

\begin{proof}
The image of the finite generating set  of $P$ in the first homogeneous component of the Lie algebra $L_p(P)$ is a finite set of generators of $L_p(P)$.  We claim that all commutators in these generators  are ad-nilpotent.
In fact, we prove that every homogeneous element $\bar a$ of $L_p(P)$
is ad-nilpotent. We may assume that $\bar a$ is the image of an element $a\in P$ in the corresponding factor  $P_{\delta (a)}/P_{\delta (a)+1}$ of the Zassenhaus filtration, where $\delta (a)$ is the degree of $a$. We fix the notation $a$ and $\bar a$ for the rest of the proof of this lemma.

For our $a\in P$, we consider the sets
$$
U_{k,s,t}=\{x\in P\mid  [[x,\,{}_ka^{s}],a^{t}]=1\}, \qquad k,s,t\in \N.
$$
Each set $U_{k,s,t}$ is closed, and
$$
P=\bigcup_{k,s,t}U_{k,s,t}
$$
by Lemma~\ref{l-ksm}. Therefore by Theorem~\ref{bct} some $U_{k,s,t}$ contains a coset $Nd$ of an open normal subgroup $N$ of $P$. We obtain that
\begin{equation}\label{e-ndpsm}
  [[nd,\,{}_ka^{s}],a^{t}]=1\qquad \text{for all}\quad n\in N.
\end{equation}
We are going to derive from this equation the desired ad-nilpotency of $\bar a$ in $L_p(P)$.

Let $s=s_1p^l$ and $t=t_1p^m$ for $s_1, t_1$ coprime to $p$. Since $\langle a^t\rangle=\langle a^{s_1p^m}\rangle$, we can replace $t$ with $s_1p^m$ in \eqref{e-ndpsm}, so that
\begin{equation}\label{e-ndpsm2}
  [[nd,\,{}_ka^{s_1p^l}],a^{s_1p^m}]=1\qquad \text{for all}\quad n\in N,
\end{equation}
where $(s_1,p)=1$. Since the image $\overline{a^{s_1}}$ of $a^{s_1}$ in  $P_{\delta (a)}/P_{\delta (a)+1}$ is equal to $s_1\bar a$ and $s_1$ is coprime to the characteristic $p$ of the ground field of $L_p$, it is sufficient to prove that $\overline{a^{s_1}}$ is ad-nilpotent. Replacing $a$ with $a^{s_1}$ we change  notation in \eqref{e-ndpsm2}, so that we have
\begin{equation}\label{e-ndpsm3}
  [[nd,\,{}_ka^{p^l}],a^{p^m}]=1\qquad \text{for all}\quad n\in N.
\end{equation}

For generators $x,y,z,h$ of a free group we write
$$
[[xy,\,{}_k z],h]=[[x,\,{}_k z],h]\cdot [[y,\,{}_k z],h]\cdot v(x,y,z,h),
$$
where the word $v(x,y,z,h)$ is a product of commutators of weight at least $k + 3$, each of which involves $x$, $y$, $h$ and involves $z$ at least $k$ times. Substituting $x=n$, $y=d$, $z=a^{p^l}$, and $h=a^{p^m}$ and using \eqref{e-ndpsm3} we obtain that
$$
 [[n,\,{}_ka^{p^l}],a^{p^m}]=v(n,d,a^{p^l},a^{p^m})^{-1}
\qquad \text{for all}\;\, n\in N.
$$
If $|P/N|=p^r$, then for any $g\in P$ we have $[g,\,{}_ra^{p^l}]\in N$, so that we also have
\begin{equation}\label{eq-ad3}
[[g,\,{}_{k+r}a^{p^l}],a^{p^m}]=v([g,\,{}_ra^{p^l}],d,a^{p^l},a^{p^m})^{-1}.
\end{equation}
We claim that $\bar a$ is ad-nilpotent in $L_p(P)$ of index $(k+r)p^l+p^m$.

Recall that $\delta  (u)$ denotes the degree of an element $u \in P$ with respect to the Zassenhaus filtration.  It is well known that
\begin{equation}\label{e-pd}
u^p\in P_{p\delta (u)}.
\end{equation}
  Furthermore, in $L_p(P)$ for the images of $u$ and $u^p$ in  $P_{\delta (u)}/P_{\delta (u)+1}$ and $P_{p\delta (u)}/P_{p\delta (u)+1}$, respectively, we have
\begin{equation}\label{e-pd2}
[x, \bar{u^p}]=[x,\,{}_p \bar{u}]
\end{equation}
(see, for example, \cite[Ch.~II, \S\,5, Exercise~10]{bou}).

By  \eqref{e-pd}  the degree of $v([g,\,{}_ra^{p^l}],d,a^{p^l},a^{p^m})$ on the right of~\eqref{eq-ad3} is at least $\delta (d)+\delta (g)+((k+r)p^l+p^m)\delta (a)$,
which is   strictly greater than $w=\delta (g)+((k+r)p^l+p^m)\delta (a)$. This means that the image of the right-hand side of \eqref{eq-ad3}  in $P_w/P_{w+1}$ is trivial. At the same time, by \eqref{e-pd2} the image of the left-hand side of \eqref{eq-ad3}  in $P_w/P_{w+1}$ is equal to the image of $[g,\,{}_{(k+r)p^l+p^m} a]$ in $P_w/P_{w+1}$, which is in turn equal to the  element    $[\bar g,\,{}_{(k+r)p^l+p^m}\bar a]$ in $L_p(P)$.
 Thus, for the corresponding homogeneous elements of $L_p(P)$ we have
 $$
 [\bar g,\,{}_{(k+r)p^l+p^m}\bar a]=0.
$$
 Since here $\bar g$ can be any homogeneous element, this means that $\bar a$ is ad-nilpotent of index $(k+r)p^l+p^m$, as claimed.
\end{proof}

\begin{lemma}\label{l-pi}
The Lie algebra $L_p(P)$ satisfies a polynomial identity.
\end{lemma}

\begin{proof}
Consider the subsets of the direct product $P\times P$
$$
V_{k,s,t}=\{(x,y)\in P\times P\mid  [[x,\,{}_ky^{s}],y^{t}]=1\}, \qquad k,s,t\in\N.
$$
Note that each subset $V_{k,s,t}$ is closed in the product topology of $P\times P$.
  By Lemma~\ref{l-ksm} we have
  $$
  P\times P=\bigcup_{k,s,t}V_{k,s,t}.
  $$
By Theorem~\ref{bct} one of the sets $V_{k,s,t}$ contains an open subset of $P\times P$. This means that there are cosets $aN$ and $bN$ of an open normal subgroup $N$ of $P$ and positive integers $k,s,t$  such that
\begin{equation*}
[[x,\,{}_ky^{s}],y^{t}]=1\qquad \text{ for any }x\in aN,\; y\in bN.
\end{equation*}
Thus, $P$ satisfies  a coset identity on cosets of a subgroup of finite index and therefore the Lie algebra $L_p(P)$ satisfies a polynomial identity by Theorem~\ref{t-coset}.
\end{proof}

We can now finish the proof of Proposition~\ref{pr-pro-p}. By Lemmas~\ref{l-ad} and \ref{l-pi} the Lie algebra  $L_p(P)$ satisfies the hypotheses of Theorem~\ref{tz}, by which  $L_p(P)$ is nilpotent. The nilpotency of the Lie algebra $L_p (P)$  of the finitely generated pro-$p $ group $P$  implies that $P$ is a $p $-adic analytic group. This result goes back to Lazard~\cite{laz}; see also \cite[Corollary~D]{sha}.
Furthermore, by a theorem of Breuillard and Gelander \cite[Theorem~8.3]{br-ge}, a  $p $-adic analytic group satisfying a coset identity on cosets of a subgroup of finite index is soluble.

Thus, $P$ is soluble, and we prove that $P$ is nilpotent by induction on the derived length of $P$.  By induction hypothesis, $P$ has an abelian normal subgroup $U$ such that $P/U$ is  nilpotent. We aim to show that $P$ is an Engel group. Since $P/U$ is nilpotent, it is sufficient to show that every element $a\in P$ is an Engel element in the product $U\langle a\rangle$. Since this product is a metabelian group, all of its elements also have countable left Engel sinks by Lemma~\ref{l-metab} and then  $U\langle a\rangle$ is nilpotent by Theorem~\ref{t-left}.

Thus, $P$ is an Engel group and therefore, being a finitely generated pro-$p$ group,  $P$ is nilpotent by Theorem~\ref{t-wz}.
\end{proof}

We now consider the general case of a  pronilpotent group.

\begin{proof}[Proof of Theorem~\ref{t2}] Let $G$ be a pronilpotent group in which every element has a countable right Engel sink; we need to prove that $G$ is locally nilpotent. By Theorem~\ref{t-wz},  it is sufficient to prove that $G$ is an Engel group.

For each prime~$p$, let $G_p$ denote the Sylow $p$-subgroup of $G$, so that $G$ is a Cartesian product of the~$G_p$, since $G$ is pronilpotent. Given any two elements $a,g\in G$, we write $g=\prod _pg_p$ and $a=\prod _pa_p$,  where $a_p,g_p\in G_p$. Clearly, $[g_q,a_p]=1$ for $q\ne p$.
   We need to show that $[g,\,{}_ma]=1$ for some positive integer $m=m(a,g)$.

  Let    $\{s_1,s_2,\dots \}$ be a countable right Engel sink of $g$.
Consider the subsets
$$
T_{i,j,k}=\{x\in \langle a\rangle \mid [g,\,{}_kx]=s_i\text{ and }  [s_i,x]=s_j\}
$$
(where $\langle a\rangle$ is the procyclic subgroup generated by $a$). Note that each $T_{i,j,k}$ is a closed subset of $\langle a\rangle$.
By the definition of a right Engel sink, we have
$$
\langle a\rangle=\bigcup _{i,j,k}T_{i,j,k}.
$$
By Theorem~\ref{bct} some $T_{i,j,k}$ contains an open subset of $\langle a\rangle$, so $Nd\subset T_{i,j,k}$ for some open subgroup $N$ of $\langle a\rangle$ and some $d\in \langle a\rangle$, so that
\begin{equation*}
[g,\,{}_kx]=s_i\quad\text{and}\quad  [s_i,x]=s_j \qquad\text{for any }  x\in Nd.
\end{equation*}
 Since $\langle a\rangle/N$ is finite, we can assume that $d=a^{s}$ for some positive integer $s$. Since $[s_i,nd]=s_j$ for all $n\in N$, it follows that $[s_i,N]=1$. Thus,
\begin{equation}\label{e-nn}
[[g,\,{}_k(n_1a^s)],n_2]=1\qquad\text{for any}\quad  n_1,n_2\in N.
\end{equation}

 Let $\sigma=\pi ( |\langle a\rangle/N|)$ be the (finite) set of prime divisors of the order of $\langle a\rangle/N$. Then $a_q\in N$ for any $q\not\in \sigma$.
Choosing $n_1=a_q^{1-s}$ and $n_2=a_q$ in \eqref{e-nn} we obtain for the $q$-components
\begin{equation}
\label{e-qk}
 [g_q,\,{}_{k+1}a_q]= [[g_q,\,{}_ka_q],a_q]=1\qquad\text{for any}\quad  q\not\in \sigma.
\end{equation}

 For every prime $p$ the group $G_p$ is locally nilpotent by Proposition~\ref{pr-pro-p},  so there is a positive integer $k_p$ such that $[g_p,\,{}_{k_p}a_p]=1$.  Now for
 $$
 m=\max\{k+1, \,\max_{p\in \sigma } \{k_p\}\}
 $$
 in view of \eqref{e-qk} we have $[g_p,\,{}_{m}a_p]=1$ for all $p$, which means that  $[g,\,{}_{m}a]=1$.
 Thus, $G$ is an Engel group and therefore it is locally nilpotent by
 Theorem~\ref{t-wz}.
\end{proof}

\section{Prosoluble groups}

Any profinite group $G$ has the largest normal pronilpotent subgroup $F(G)$,  called the Fitting subgroup of $G$. Further terms of the Fitting series are defined by induction: $F_1(G)=F(G)$ and $F_{i+1}(G)$ is the inverse image of $F(G/F_i(G))$. By definition a group has finite Fitting height if $F_k(G)=G$ for some $k\in {\mathbb N}$.

Recall that by Theorem~\ref{t2} any pronilpotent group with countable right Engel sinks is locally nilpotent. Therefore, if $G$ is a profinite group with countable right Engel sinks, then the Fitting subgroup $F(G)$  is  locally nilpotent.

The proof of the main Theorem~\ref{t-main} for a prosoluble group $G$ with countable right Engel sinks will follow from a key proposition stating that $F(G)\ne 1$. We approach the proof of this proposition in a number of steps. First we list several profinite  analogues of the properties of coprime automorphisms of finite groups, which are used in this section in relation to Engel sinks.

 If $\varphi$  is an automorphism of a finite group $H$ of coprime order, that is, such that $(|\varphi |,|H|)=1$, then  we say for brevity that $\varphi$ is a coprime automorphism of~$H$. This definition is extended to profinite groups as follows.
We say that $\varphi$ is a \textit{coprime automorphism}  of a profinite group $H$  meaning that a procyclic group $\langle\varphi\rangle$ faithfully acts on $H$ by continuous automorphisms   and $\pi (\langle \varphi\rangle)\cap \pi (H)=\varnothing$. Since the semidirect product $H\langle \varphi\rangle$ is also a profinite group, $\varphi$ is a coprime automorphism of $H$
 if and only if for every open normal $\varphi$-invariant subgroup $N$ of $H$ the automorphism (of finite order) induced by $\varphi$ on $H/N$ is a coprime automorphism.
The following lemma is derived from an analogue of the Schur--Zassenhaus theorem for profinite groups; we shall freely use this fact without special references.

\begin{lemma}[{see \cite[Lemma~4.1]{khu-shu191}}]\label{l-inv}
If $\varphi$ is a coprime automorphism of a profinite group $G$, then for every prime $q\in \pi (G)$ there is a $\varphi$-inva\-ri\-ant Sylow $q$-subgroup of $G$. If $G$ is in addition prosoluble, then for every subset $\sigma\subseteq \pi (G)$ there is a $\varphi$-invariant Hall $\sigma$-subgroup of~$G$.
\end{lemma}

The following lemma is a special case of \cite[Proposition~2.3.16]{rib-zal}.

\begin{lemma} \label{l-cover}
If $\varphi $ is a coprime automorphism of a profinite group $G$ and $N$ is a $\varphi$-inva\-ri\-ant closed normal subgroup of $G$, then every fixed point of $\varphi$ in $G/N$ is an image of a fixed point of $\varphi$ in $G$, that is, $C_{G/N}(\varphi)=C(\varphi )N/N$.
\end{lemma}

As a consequence, we have the following.

\begin{lemma}[{see \cite[Lemma~4.3]{khu-shu191}}] \label{l-gff}
If $\varphi $ is a coprime automorphism of a profinite group $G$, then
$[[G,\varphi],\varphi]=[G,\varphi]$.
\end{lemma}

We will be applying the following profinite version of a theorem of Thompson.  Namely, Thompson \cite{th} proved that if $G$ is a finite soluble group on which a finite soluble group $A$ of coprime order acts by automorphisms, then the Fitting height $h(G)$ is bounded in terms of $h(C_G(A))$ and the number of prime divisors of $|A|$ counting multiplicities. (Further results in this direction were devoted to improving the corresponding bounds, with best possible one obtained by Turull, see his survey \cite{turull}.) The profinite version of Thompson's theorem can be deduced by standard arguments in the spirit of \cite[Lemma~2]{wil83}, since the hypotheses are inherited by quotients by closed normal subgroups by Lemma~\ref{l-cover}.

\begin{theorem}[Thompson]\label{thompson}
Let $G$ be a prosoluble group on which a finite soluble group $A$ acts by continuous automorphisms and suppose that $\pi(G)\cap \pi(A)=\varnothing$. If $C_G(A)$ has finite Fitting height, then $G$  also has finite Fitting height.
\end{theorem}

We begin our step-by-step approach to proving that $F(G)\ne 1$ in any nontrivial prosoluble group with countable right Engel sinks. The first step is considering the case where all Sylow subgroups are finite.

 \begin{lemma}\label{l-finsyl}
 Suppose that $G$ is a nontrivial prosoluble group in which Sylow $p$-subgroups are finite for all primes $p$. If every element of $G$ has a countable right Engel sink, then $F(G)\ne 1$.
 \end{lemma}

\begin{proof}
For a given arbitrary element $g\in G$, let $\{s_1,s_2,\dots \}$ be a countable right Engel sink of~$g$. Consider the subsets
$$
T_{i,j,k}=\{x\in G \mid [g,\,{}_kx]=s_i\text{ and }  [s_i,x]=s_j\}.
$$
Note that each $T_{i,j,k}$ is a closed subset of $G$.
By the definition of a right Engel sink, we have
$$
G=\bigcup _{i,j,k}T_{i,j,k}.
$$
By Theorem~\ref{bct} some $T_{i,j,k}$ contains a coset $Nd$ of an open normal subgroup $N$ of $G$, so that
\begin{equation}\label{e-gss}
[g,\,{}_kx]=s_i\quad\text{and}\quad  [s_i,x]=s_j \quad\text{for any }  x\in Nd.
\end{equation}
Since $[s_i,nd]=s_j$ for all $n\in N$, it follows that $[s_i,N]=1$. Then $s_i$ has centralizer of finite index in $G$, and $s_i$ has finitely many conjugates in $G$. Hence the normal closure $\langle s_i^G\rangle$ is central-by-finite.
Given that $G$ is a prosoluble group, it follows that if $s_1\neq1$, then $G$ has a normal abelian subgroup and therefore also a closed normal abelian subgroup, so that $F(G)\ne 1 $.

Thus, the lemma is proved unless for every $g\in G$ the element $s_i$ given by the above argument in \eqref{e-gss} is trivial. In other words, it remains to consider the case where for every $g\in G$ there is a positive integer $k=k(g)$ and  a coset $Nd$ of an open subgroup $N$ of $G$ such that
\begin{equation}\label{e-gss2}
[g,\,{}_kx]=1 \quad\text{for any } x\in Nd.
\end{equation}

We now observe that the group $G$ has only countably many open normal subgroups. Indeed, any normal subgroup of finite index $n$ must contain the normal subgroup $H({\pi (n))}$ generated by all Sylow $q$-subgroups for all primes $q\not\in\pi(n)$, and each quotient by $G/H({\pi (n))}$ is finite, since all Sylow subgroups of $G$ are finite. Therefore $G$ has only finitely many normal subgroups of any given finite index and hence only countably many open normal subgroups.
Consequently, $G$ has countably  many cosets of such subgroups, say, $\{D_1,D_2,\dots \}$.
Consider the subsets
$$
U_{j,k}=\{x\in G \mid [x,\,{}_ky]=1\text{ for any }  y\in D_j\}.
$$
Note that each $U_{j,k}$ is a closed subset of $G$.
By our assumption involving \eqref{e-gss2}, we have
$$
G=\bigcup _{j,k}U_{j,k}.
$$
By Theorem~\ref{bct} some $U_{j,k}$ contains a coset $Kb$ of an open normal subgroup $K$ of $G$, so that
\begin{equation*}
 [x,\,{}_ky]=1\quad\text{ for any }  x\in Kb\text{ and any }y\in D_j,
\end{equation*}
where $D_j=Nd$ for some $d\in G$ and an open normal subgroup $N$ of $G$. Setting $L=K\cap N$  we obtain
\begin{equation}\label{e-gss3}
 [ub,\,{}_kvd]=1\quad\text{ for any }  u,v\in L.
\end{equation}
Let $\s=\pi (G/L)$, which is a finite set of primes. Then $G=LH$ for a  Hall $\s$-subgroup $H$, which is finite, since all Sylow subgroups are finite. We can choose the coset representatives $b,d\in H$ satisfying \eqref{e-gss3}. Then $B=\langle b , d \rangle$ is a finite  subgroup of $H$.

By standard commutator formulae equation \eqref{e-gss3} implies that
\begin{equation}
 1=[ub,\,{}_kvd]=[b,\,{}_k d][u,\,{}_kv]=[u,\,{}_kv]\quad\text{ for any }  u,v\in C_L(B).
\end{equation}
Hence  $C_L(B)$ is a $k$-Engel group. Therefore $C_L(B)$ is locally nilpotent by Theorem~\ref{t-wz}.

Recall that $\s=\pi (G/L)$ and $H$ is a Hall $\s$-subgroup. Since $H$  is finite, there is an open normal subgroup $M$ of $G$ that intersects $H$ trivially; replacing $L$ with $L\cap M$, we can assume that $L$ is a $\s '$-subgroup.

By Theorem~\ref{thompson} applied to the action of $B$ on $L$ we obtain that $L$ has a finite series of characteristic closed subgroups with pronilpotent factors. This implies that $F(G)\ne 1$ (even if $L=1$, since $G$ is prosoluble).
\end{proof}

We now consider coprime automorphisms in relation to Engel sinks. In the proof of the following lemma we use the well-known fact that if $G$ is nilpotent and $G/G'$ is finite, then $G$ is finite  (see, for example, \cite[5.2.6]{rob}).

 \begin{lemma}\label{l-copr3}
 Let $\varphi$ be a coprime automorphism of a pronilpotent
group~$G$. If  all elements of  the semidirect product $G\langle \varphi\rangle$ have countable right Engel sinks, then $\gamma _{\infty} (G\langle \varphi\rangle)$ is finite  and  $\gamma _{\infty}(G\langle \varphi\rangle)= [G,\varphi]$.
 \end{lemma}

 \begin{proof}
The group $G$ is locally nilpotent by Theorem~\ref{t2}.  The quotient $G\langle \varphi\rangle/[G,\varphi]$ is obviously the direct product of the images of $G$ and $\langle \varphi\rangle$ and therefore is pronilpotent. Hence,  $\gamma _{\infty}(G\langle \varphi\rangle)\leq [G,\varphi]$.  By Lemma~\ref{l-gff},
\begin{equation*}
[[G,\varphi],\varphi]=[G,\varphi].
\end{equation*}
Therefore also  $\gamma _{\infty}(G\langle \varphi\rangle)\geq [G,\varphi]$, so that  $\gamma _{\infty}(G\langle \varphi\rangle)= [G,\varphi]$.

Let $V$ be the quotient of  $[G,\varphi ]$ by its derived subgroup. The semidirect product $V\langle \varphi\rangle$ is metabelian and therefore all elements of it also have countable left Engel sinks.  By Theorem~\ref{t-left} then  $\gamma _{\infty}(V\langle \varphi\rangle)= [V,\varphi]=V$ is finite. It follows that the locally nilpotent pronilpotent group $[G,\varphi ]$ is finitely generated and therefore nilpotent and finite.
 \end{proof}

We shall further need the following simple lemma about finite groups.

\begin{lemma}\label{l-fq}
Suppose that $G$ is a finite $q$-soluble group admitting a coprime automorphism $\alpha$ such that $G=[G,\alpha ]$. If $Q$ is a nontrivial $\alpha$-invariant Sylow $q$-subgroup of $G$, then $[Q,\alpha ]\ne 1$.
\end{lemma}

\begin{proof}
We can assume from the outset that $O_{q'}(G)=1$; then $O_q(G)\ne 1$ and $C_G(O_q(G))\leq O_q(G)$. Suppose the opposite: $[Q,\alpha ]=1$. Then $[O_q(G),\alpha ]=1$, whence $[O_q(G),[\alpha ,G] ]=[O_q(G),G]=1$ by Lemma~\ref{l-fng}, so that $G=O_q(G)$ and $G=[G,\alpha ]=[Q,\alpha]=1$, a contradiction.
\end{proof}

We now consider the action of a coprime automorphism of a prosoluble
group.

 \begin{lemma}\label{l-copr4}
 Let $\varphi$ be a nontrivial coprime
 automorphism
  of a prosoluble
group~$G$. If  all elements of  the semidirect product $G\langle \varphi\rangle$ have countable right Engel sinks, then $F([G, \f])\ne\nobreak 1$.
 \end{lemma}

\begin{proof}
By Lemma~\ref{l-gff} we can assume from the outset that $G=[G,\f ]$.

For a prime $q$, let $Q$ be a $\f$-invariant Sylow $q$-subgroup of $G$. Then $[Q,\f ]$ is finite by Lemma~\ref{l-copr3}. Since $G\langle\f\rangle$ is a profinite group, it has an open normal subgroup $U$ such that $U\cap [Q,\f ]=1$; then  $N=G\cap U$ is a $\f$-invariant open normal subgroup of $G$ such that $N\cap [Q,\f ]=1$.

We claim that  $[N,\f ]$ is a $q'$-group. Indeed, $Q\cap N$ is a Sylow $q$-subgroup of $N$, and then $Q_1=Q\cap [N,\f ]$ is a Sylow $q$-subgroup of  $[N,\f ]$. We have $[Q_1,\f]\leq [Q,\f]\cap N=1$. Applying Lemma~\ref{l-fq} to every finite quotient of $[N,\f ]\langle\f\rangle$, we obtain that  $[N,\f ]$ is a $q'$-group.

Then $[N,\f ]\leq O_{q'}(N)\leq O_{q'}(G)$. In the quotient $\bar G=G/O_{q'}(G)$ we have $[\bar N,\f]=1$, 
so $\bar N\leq Z(\bar G)$ by Lemma~\ref{l-fng}. This means that $\bar G$ is central-by-finite, whence $\bar G'$ is finite by Schur's theorem \cite[10.1.4]{rob}. Since  $G/G'=[G/G',\f]$ is also finite by Lemma~\ref{l-copr3}, we obtain that $G/O_{q'}(G)$ is finite. Therefore a Sylow $q$-subgroup of $G$ is finite for every prime $q$. The result now follows from Lemma~\ref{l-finsyl}.
\end{proof}

We now consider the case of prosoluble groups of finite Fitting height.

 \begin{lemma}\label{l-height}
 Let $G$ be a prosoluble group of finite Fitting height. If  every element of  $G$ has a countable right Engel sink, then $\gamma _{\infty} (G)$ is finite.
 \end{lemma}

\begin{proof}  It is sufficient to prove the result for the case of Fitting height 2. Then the general case will follow by induction on the Fitting height $k$ of $G$. Indeed, then $\gamma _{\infty}(G/ \gamma _{\infty}(F_{k-1}(G)))$ is finite, while $\gamma _{\infty}(F_{k-1}(G))$ is finite by the induction hypothesis, and as a result, $\gamma _{\infty}(G)$ is finite.

Thus, we assume that $G=F_2(G)$. By Theorem~\ref{t-left}, it is sufficient to show that every element  $a\in G$ has a finite left Engel sink.  Since $G/F(G)$ is locally nilpotent, a left Engel sink of $a$ in $F(G)\langle a\rangle$ is also a left Engel sink of $a$ in $G$. Therefore it is sufficient to prove that $\gamma _{\infty}(F(G)\langle a\rangle)$ is finite.

For a prime $p$, let $F_p$ be a Sylow $p$-subgroup of $F(G)$, and write $a=a_pa_{p'}$, where $a_p$ is a $p$-element, $a_{p'}$ is a $p'$-element, and $[a_p,a_{p'}]=1$. Then $F_p\langle a_p\rangle$ is a normal Sylow $p$-subgroup of  $F_p\langle a\rangle$, on which $a_{p'}$ induces by conjugation a coprime automorphism. By Lemma~\ref{l-copr3} the subgroup $\gamma _{\infty}(F_p\langle a\rangle)=[F_p,a_{p'}]$ is finite.
Since
$$
\gamma _{\infty}(F(G)\langle a\rangle)=\prod_p\gamma _{\infty}(F_p\langle a\rangle)=\prod_p [F_p,a_{p'}],
$$
it remains to prove that $\gamma _{\infty}(F_p\langle a\rangle)=[F_p,a_{p'}]=1$ for all but  finitely many primes $p$.

Let $V=F(G)/[F(G),F(G)]$. In the metabelian group $V\langle a\rangle$ all elements also have countable left Engel sinks by Lemma~\ref{l-metab}. By Theorem~\ref{t-left} the subgroup $\gamma _{\infty}(V\langle a\rangle)$ is finite. Since
$$
[F(G),F(G)] =\prod_p [F_p,F_p],
$$
it follows that $[F_p,a_{p'}]\leq [F_p,F_p]$ for all but finitely many primes $p$. But if $[F_p,a_{p'}]\leq [F_p,F_p]$, then $C_{F_p}(a_{p'})[F_p,F_p]=F_p$ by Lemma~\ref{l-cover}, whence $C_{F_p}(a_{p'})=F_p$, that is, $[F_p,a_{p'}]=\nobreak 1$. Hence the result.
\end{proof}

 \begin{lemma}\label{l-copr5}
 Let $\varphi$ be a coprime automorphism
 of a prosoluble group~$G$. If  all elements of  the semidirect product $G\langle \varphi\rangle$ have countable right Engel sinks, then $[G,\f]$ is finite.
 \end{lemma}

\begin{proof}
  By Lemma~\ref{l-gff} we can assume from the outset that $G=[G,\f]$.
  Since $\g_{\infty}(F_2(G))$ is finite by Lemma~\ref{l-height}, there is a $\f$-invariant open normal subgroup $N$ of $G$ such that $N\cap \g_{\infty}(F_2(G))=1$. It follows that $\g_{\infty}(F_2(N))=1$, which means that $F_2(N)=F(N)$. Then $\f$ must act trivially on $N/F(N)$, since otherwise   $F_2(N)\ne F(N)$ by Lemma~\ref{l-copr4} applied to $N/F(N)$. Thus, $[N,\f ]\leq F(N)\leq F(G)$. Then $NF(G)/F(G)\leq Z(G/F(G))$ by Lemma~\ref{l-fng}, since $G=[G,\f]$ by our assumption. In particular, the Fitting height of $G$ is finite, and we obtain that $[G,\f]$ is finite by Lemma~\ref{l-height}.
  \end{proof}

We now prove the key proposition of this section.

\begin{proposition}\label{pr-f}
If every element of a nontrivial prosoluble group $G$ has a countable right Engel sink, then $F(G)\ne 1$.
\end{proposition}

\begin{proof}
Since $G$ is prosoluble, we have $G\ne
\gamma _{\infty}(G)$. Choose a prime $p\in \pi (G/\gamma _{\infty}(G))$. If $\gamma _{\infty}(G)$ is a pro-$p$  group, then $F(G)\ne 1$ and we are done. Otherwise, let $H$ be a Hall $p'$-subgroup of $\gamma _{\infty}(G)$. 
By an analogue of the Frattini argument we have $G=N_G(H)\gamma _{\infty}(G)$. Indeed, for any $x\in G$ the Hall $p'$-subgroups $H$ and $H^x$ of $\gamma _{\infty}(G)$ are conjugate in $\gamma _{\infty}(G)$, so that $H=H^{xy}$ for $y\in \gamma _{\infty}(G)$ and then $x\in N_G(H)y^{-1}$.

We can now choose a nontrivial $p$-element $a\in N_G(H)$ (so that $|a|=p^k$, where $k\in \N\cup \{\infty\}$). By Lemma~\ref{l-copr5} the subgroup $[H,a]$ is finite. Therefore there is an open normal $a$-invariant subgroup $N$ of $\gamma _{\infty}(G)$ such that $N\cap [H,a]=1$.
Then $[N\cap H,a]=1$.
Since $N=(H\cap N)P$ for an $a$-invariant Sylow $p$-subgroup $P$ of $\g_{\infty}(G)$, we have $[N,a]=[(H\cap N)P,a]=[P,a]$. Hence the subgroup $[N,a]$ is pronilpotent,  and therefore,
$$
[N,a]\leq F(N)\leq F(\gamma _{\infty}(G))\leq F(G).
 $$
 Thus, the proposition is proved if $[N,a]\ne 1$.

If $[N,a]=1$, then also $[N,[\gamma _{\infty}(G),a]]=1$. Then $[\gamma _{\infty}(G),a]$ has a central subgroup of finite index and therefore has finite Fitting height. By Lemma~\ref{l-height}, $\gamma _{\infty}([\gamma _{\infty}(G),a])$ is finite, and therefore $F([\gamma _{\infty}(G),a])\ne 1$ unless $[\gamma _{\infty}(G),a]=1$.
Since
$$
F([\gamma _{\infty}(G),a])\leq F(\gamma _{\infty}(G))\leq F(G),
$$
the proof is complete if  $[\gamma _{\infty}(G),a]\ne 1$.
 Finally, if $[\gamma _{\infty}(G),a]=1$, then $a$ is an Engel element since $G/\gamma _{\infty}(G)$ is locally nilpotent by Theorem~\ref{t2}. Then the normal subgroup  $[G,a]\langle a\rangle$ is pronilpotent by Baer's theorem \cite[Satz~III.6.15]{hup}, and $F(G)\ne 1$.
\end{proof}

We are now ready to prove the main result of this section.

\begin{theorem}\label{t3}
Suppose that $G$ is a prosoluble group in which every element has a countable right Engel sink. Then $G$ has a finite normal subgroup $N$ such that $G/N$ is locally nilpotent.
\end{theorem}

\begin{proof}
By Theorem~\ref{t2} it is sufficient to prove that $\gamma _{\infty}(G)$ is finite.
Since $\g_{\infty}(F_2(G))$ is finite by Lemma~\ref{l-height}, there is an open normal subgroup $N$ of $G$ such that $N\cap \g_{\infty}(F_2(G))=1$. It follows that $\g_{\infty}(F_2(N))=1$, which means that $F_2(N)=F(N)$.  It follows from  Proposition~\ref{pr-f} that $F(N)=N$ is locally nilpotent, so that $N\leq F(G)$. Hence the quotient group $G/F(G)$ is finite, and therefore the Fitting height of $G$ is finite. We obtain that $\g_{\infty}(G)$ is finite by Lemma~\ref{l-height}.
\end{proof}

Here we also derive the following corollary for a virtually prosoluble group (that is, a group with a prosoluble open normal subgroup), which will be needed in the sequel.

\begin{corollary}\label{c-virt}
Suppose that $G$ is a virtually prosoluble group in which every element has a countable right Engel sink. Then $G$ has a finite normal subgroup $N$ such that $G/N$ is locally nilpotent.
\end{corollary}

\begin{proof}
By Theorem~\ref{t2} it is sufficient to show that $\gamma _{\infty}(G)$ is finite.
By hypothesis, $G$ has an open normal prosoluble subgroup $H$.
 By Theorem~\ref{t3}, $\gamma _{\infty}(H)$ is finite. Therefore, passing to the quotient group,  we can assume that $\gamma _{\infty}(H)=1$ and the Fitting subgroup
$F(G)$ is open.

Since $G/F(G)$ is finite, we can use induction on $|G/F(G)|$.
       The basis of this induction includes the trivial case $G/F(G)=1$ when $\gamma _{\infty}(G)=1$. But the bulk of the proof deals with the case where $G/F(G)$ is a finite simple group.  If $G/F(G)$ is abelian, then $G$ has Fitting height 2 and  $\gamma _{\infty }( G)$ is finite by Lemma~\ref{l-height} and the proof is complete.

Thus, suppose that $G/F(G)$ is a non-abelian finite simple group.
Let $p$ be a prime divisor of $|G/F(G)|$, and $g\in G\setminus F(G)$ an element of order $p^n$, where $n$ is either a positive integer or $\infty$ (so $p^n$ is a Steinitz number). Let $T$ be the Hall $p'$-subgroup of $F(G)$. By Lemma~\ref{l-copr3} the subgroup $[T , g]$ is finite.

Since $[T, g]$ is normal in $F(G)$, its normal closure $R=\langle [T, g]^G\rangle $ in $G$ is a product of finitely many conjugates and is therefore also finite.
Therefore it is sufficient to prove that $\gamma _{\infty }(G/R)$ is finite. Thus, we can assume that $R=1$. Note that then $[T, g^a]=1$ for any conjugate $g^a$ of $g$.

Choose a transversal $\{u_1,\dots, u_k\}$ of $G$ modulo $F(G)$.
 Let $G_1=\langle g^{u_1}, \dots ,g^{u_k}\rangle$. Clearly, $G_1F(G)/F(G)$ is generated by the conjugacy class of the image of $g$. Since $G/F(G)$ is simple, we have $G_1F(G)=G$. By our assumption, the Hall $p'$-subgroup $T$ of $F(G)$
 is centralized by all elements $g^{u_i}$. Hence,  $[G_1, T]=1$. Let $P$ be the Sylow $p$-subgroup of $F(G)$ (possibly, trivial). Then also $[PG_1, T]=1$, and therefore
 $$\gamma _{\infty }(G)=\gamma _{\infty }(G_1F(G))= \gamma _{\infty }(PG_1).$$

Let the bar denote images in $\bar G=G/P$. Note that $\gamma _{\infty}(\bar G)=\gamma _{\infty}(\bar G_1)$, while $F(\bar G)=\bar T$ and $\bar G/\bar T=\bar G_1\bar T/\bar T\cong F/F(G)$ is a non-abelian finite simple group. Hence, $\bar G=\gamma _{\infty}(\bar G_1)\bar T$. Therefore, since $[\gamma _{\infty}(\bar G_1), \bar T]=1$,
$$
\gamma _{\infty}(\bar G_1)=[\gamma _{\infty}(\bar G_1), \bar G_1]=[\gamma _{\infty}(\bar G_1),\gamma _{\infty}(\bar G_1)\bar T]=[\gamma _{\infty}(\bar G_1),\gamma _{\infty}(\bar G_1)].
$$
As a result, $\gamma _{\infty}(\bar G_1)\cap \bar{T}$ is contained both in the centre and the derived subgroup of $\gamma _{\infty}(\bar G_1)$, and therefore is isomorphic to a subgroup of the Schur multiplier of the finite group $\gamma _{\infty}(\bar G_1)/ (\gamma _{\infty}(\bar G_1)\cap \bar{T})\cong G/F(G)$. Since the Schur multiplier of a finite group is finite \cite[Hauptsatz~V.23.5]{hup}, we obtain that $\gamma _{\infty}(\bar G_1)\cap \bar{T}$ is finite. Since $\bar T$ is canonically isomorphic to $T$, it follows that
$$
\gamma _{\infty }(G)\cap T\cong\gamma _{\infty }(\bar G)\cap \bar T=\gamma _{\infty }(\bar G_1)\cap \bar T
$$
is also finite. Therefore  we can assume that $T=1$, in other words, that $F(G)$ is a $p$-group.

Since $G/F(G)$ is a non-abelian simple group, we can choose another prime $r\ne p$ dividing $|G/F(G)|$ and repeat the same arguments as above with $r$ in place of $p$. As a result, we reduce the proof to the case $F(G)=1$, where the result is obvious.

We now finish the proof of Corollary~\ref{c-virt} by induction on  $|G/F(G)|$. The basis of this induction where $G/F(G)$ is a simple group was proved above. Now suppose that $G/F(G)$ has  a nontrivial proper normal subgroup with full inverse image $N$, so that $F(G)<N\lhd G$. Since $F(N)=F(G)$, by induction applied to $N$ the group $\gamma _{\infty }(N)$ is finite. Since $N/\gamma _{\infty }(N)\leq F( G/\gamma _{\infty }(N))$,  by induction applied to $G/\gamma _{\infty }(N)$ the group $ \gamma _{\infty }(G/\gamma _{\infty }(N) )$ is also finite. As a result, $\gamma _{\infty }(G) $ is finite, as required.
\end{proof}

\section{Profinite groups}

We approach the general case of profinite groups by obtaining bounds for the so-called nonprosoluble length. These bounds follow from the bounds for nonsoluble length of the corresponding finite quotients. We begin with the relevant definitions.

 The  \textit{nonsoluble length} $\lambda (H)$  of a finite group $H$ is defined as the minimum number of nonsoluble factors in a normal series in which every  factor  either is soluble or is a direct product of non-abelian simple groups. (In particular, the group is soluble if and only if its nonsoluble length is $0$.) Clearly, every finite group has a normal series with these properties, and therefore its nonsoluble length is well defined.  It is easy to see that the nonsoluble length $\lambda (H)$ is equal to the least positive integer $l$ such that there is a series of characteristic subgroups
\begin{equation*}
1=L_0\leqslant R_0 <  L_1\leqslant R_1<  \dots \leqslant R_{l}=H
\end{equation*}
in which each quotient $L_i/R_{i-1}$ is a (nontrivial) direct product of non-abelian simple groups, and each quotient $R_i/L_{i}$ is soluble (possibly trivial).

We shall use the following result of Wilson \cite{wil83}, which we state in the special case of $p=2$ using the terminology of nonsoluble length.

\begin{theorem}[{see \cite[Theorem~2*]{wil83}}]\label{t-wil83}
Let $K$ be a normal subgroup of a finite group~$G$. If a Sylow $2$-subgroup $Q$ of $K$ has a coset $tQ$ consisting of elements of order dividing $2^k$, then the nonsoluble length of $K$ is at most $k$.
\end{theorem}

We now turn to profinite groups. It is natural to say that a profinite group $G$ has finite \textit{nonprosoluble length} at most $l$ if $G$ has a normal series \begin{equation*}
1=L_0\leqslant R_0 <  L_1\leqslant R_1<  \dots \leqslant R_{l}=G
\end{equation*}
in which each quotient $L_i/R_{i-1}$ is a (nontrivial) Cartesian product of non-abelian finite simple groups, and each quotient $R_i/L_{i}$ is prosoluble (possibly trivial).
 As a special case of a general result in Wilson's paper \cite{wil83} we have the following.

\begin{lemma}[{see \cite[Lemma~2]{wil83}}]\label{l-nsl}
If, for some positive integer $m$, all continuous finite quotients of a profinite group~$G$ have nonsoluble length at most $m$, then $G$ has finite nonprosoluble length at most~$m$.
\end{lemma}

We  now prove a key proposition on bounds for the nonprosoluble length.

\begin{proposition}\label{pr-fnl}
Suppose that $G$ is a profinite group
in which every element has a countable right Engel sink. Then $G$ has finite nonprosoluble length.
\end{proposition}

\begin{proof}
Let $ H=\bigcap G^{(i)} $
be the intersection of the derived series of $G$ (where $G^{(1)}=[G,G]$ and by induction $G^{(i+1)}=[G^{(i)},G^{(i)}]$).
Then $H=[H,H]$. Indeed,
if $H\ne [H,H]$, then the quotient $G/[H,H]$ is a prosoluble group by Lemma~\ref{l-prosol-by-prosol}, whence $\bigcap G^{(i)}=H\leq [H,H]$, a contradiction.
Since the quotient $G/H$ is prosoluble, it is sufficient to prove the proposition for $H$. Thus, we can assume from the outset that $G=[G,G]$.

Let $T$ be a Sylow $2$-subgroup of $G$. By Theorem~\ref{t2} the group $T$ is locally nilpotent. Consider the subsets of the direct product $T\times T$
$$
S_{i}=\{(x,y)\in T\times T\mid \text{the subgroup }\langle x,y\rangle\text{ is nilpotent of class at most }i\}.
$$
Note that each subset $S_{i}$ is closed in the product topology of $T\times T$, because the condition defining $S_i$ means that all commutators of weight $i+1$ in $x,y$ are trivial. Since every $2$-generator subgroup of $T$ is nilpotent, we have
$$
\bigcup _iS_{i}=T\times T.
$$
By
Theorem~\ref{bct} one of the sets $S_i$ contains an open subset of $T\times T$. This means that there are cosets $aN$ and $bN$ of an open normal subgroup $N$ of $T$ and a positive integer $c$  such that
\begin{equation}\label{e-2nilp}
\langle x,y\rangle\text{ is nilpotent of class }c\text{ for any }x\in aN,\; y\in bN.
\end{equation}
(The subsequent arguments include the case where $N=T$, even with certain simplifications.)

Let $K$ be an open normal subgroup of $G$ such that $K\cap T\leq N$. If we replace $N$ by $K\cap T$, then \eqref{e-2nilp} still holds with the same $a,b$. Hence we can assume that $N$ is a Sylow $2$-subgroup of $K$.

We now apply the following general fact (which, for example, immediately follows from \cite[Lemma 2.8.15]{rib-zal}).

\begin{lemma}\label{l-dop} Let $G$ be a profinite group and $K$ a normal open subgroup of $G$. There exists a subgroup $H$ of $G$ such that
$G=KH$ and $K\cap H$ is pronilpotent.
\end{lemma}

Let $H$ be the subgroup given by this lemma for our group $G$ and subgroup $K$. Since $H$ is virtually pronilpotent and every element has a countable right Engel sink, by Corollary~\ref{c-virt} the subgroup $\gamma _\infty(H)$ is finite. Recalling our assumption that $G=[G,G]$, we obtain
$$
G=[G,G]=\gamma _{\infty}(G)\leq \gamma _{\infty}(HK)\leq \gamma _{\infty}(H)K.
$$
Thus, $G=\gamma _{\infty}(H)K$, where $\gamma _{\infty}(H)$ is a finite subgroup.

Hence we can choose the coset representative $a$ satisfying \eqref{e-2nilp}
in a conjugate of a Sylow $2$-subgroup of $\gamma _{\infty}(H)$, and therefore having finite order, say, $|a|=2^n$.

For any $y\in bN$ the $2$-subgroup $\langle a,y\rangle$ is nilpotent of class at most $c$, while  $a^{2^n}=1$. Then
\begin{equation}\label{e-ay2nc}
[a,y^{2^{n(c-1)}}]=1.
 \end{equation}
 This follows from well-known commutator formulae (and for any $p$-group); see, for example, \cite[Lemma~4.1]{shu00}.

In particular,  for any $z\in N$ by using \eqref{e-ay2nc} we obtain
\begin{equation}\label{e-2eng2}
[z,\,{}_c y^{2^{n(c-1)}}]=[az,\,{}_c y^{2^{n(c-1)}}]=1,
\end{equation}
since $\langle az,y^{2^{n(c-1)}}\rangle$ is  a subgroup of  $\langle az,y\rangle$, which is nilpotent of class $c$  by \eqref{e-2nilp}. (Note that in the case $N=T$ we could have put $K=G$, $H=1$, $a=b=1$, and $n=0$.)

Our aim is to show that there is a uniform bound, in terms of $|G:K|$, $c$, and $n$,  for the nonsoluble length of all continuous finite quotients of $G$.
Let $M$ be an open normal subgroup of $G$ and let the bar denote the images in $\bar G=G/M$. It is clearly sufficient to obtain a required bound for the nonsoluble length of $\bar K$.

Let $R_0$ be the soluble radical of $\bar K$, and $L_1$ the inverse image of the generalized Fitting subgroup of $\bar K/R_0$, so that
\begin{equation}\label{e-soc}
L_1/R_0=S_1\times S_2\times \dots\times S_k
\end{equation}
is a direct product of non-abelian finite simple groups. Note that $R_0$ and $L_1$ are normal subgroups of $\bar G$. The group $\bar G$ acting by conjugation induces a permutational action on the set $\{S_1,S_2,\dots ,S_k\}$. The kernel of the restriction of this permutational action to $\bar K$ is contained in the inverse image $R_1$ of the soluble radical of $\bar K/L_1$:
\begin{equation}\label{e-soc2}
\bigcap _iN_{\bar K}(S_i)\leq R_1.
\end{equation}
This follows from the validity of Schreier's conjecture on the solubility of the outer automorphism groups of non-abelian finite simple groups, confirmed by the classification of the latter, because $L_1/R_0$ contains its centralizer in $\bar K/R_0 $.

Let $e$ be the least positive integer such that $2^{e}\geq c$, and let $t= 2^{n(c-1)+e}$. We claim that for any   $y\in \bar b\bar N$ the element  $y^{2^t}$ normalizes each factor $S_i$ in \eqref{e-soc}.
Arguing by contradiction, suppose that the element $y^{2^t}$ has a nontrivial orbit on the set of the  $S_i$. Then the element $y^{2^{n(c-1)}}$ has an orbit of length $2^s\geq  2^{e+1}$ on this set; let $\{T_1,T_2,\dots , T_{2^s}\}$ be such an orbit cyclically permuted by $y^{2^{n(c-1)}}$.
Since non-abelian finite simple groups have even order (by the Feit--Thompson theorem \cite{fei-tho}) and the subgroups $S_i$ are subnormal in $\bar K/R_0$, each subgroup $S_i$ contains a nontrivial element of $\bar NR_0/R_0$. If $x$ is a nontrivial element of $T_1\cap  \bar NR_0/R_0$, then the commutator
$$
[x,\,{}_c \bar y^{2^{n(c-1)}}],
$$
written as an element of $T_1\times T_2\times \dots\times  T_{2^s}$, has a nontrivial component in $T_{c+1}$ since $2^s\geq  2^{e+1}>c$. This, however, contradicts~\eqref{e-2eng2}.

Thus, for any element  $y\in \bar b\bar N$ the power $y^{2^t}$ normalizes each factor $S_i$ in \eqref{e-soc}. Let $2^{d}$ be the highest power of $2$ dividing $|G:K|$, and let $u=\max\{t,d\}$. Then $y^{2^u}\in R_1$ by \eqref{e-soc2}, since $y^{2^u}\in \bar K$ and $y^{2^u}$ normalizes each $S_i$ in \eqref{e-soc} by the choice of $u$.

As a result, in the quotient $\bar G/R_1$ all elements of the coset $\bar b\bar NR_1/R_1$ of the Sylow $2$-subgroup $\bar NR_1/R_1$ of $\bar K/R_1$ have order dividing $2^u$. We can now apply  Theorem~\ref{t-wil83}, by which the nonsoluble length of $\bar K/R_1$ is at most $u$. Then the  nonsoluble length of $\bar K$ is at most $u+1$. Clearly, the nonsoluble length of $\bar G/\bar K$ is bounded in terms of  $|G:K|$.
As a result, since the number $u$ depends only on $|G:K|$, $n$, and $c$, the nonsoluble length of $\bar G$ is bounded in terms of these parameters only. Since this  holds for any continuous finite quotient of the profinite group $G$,  the group $G$ has finite nonprosoluble length by Lemma~\ref{l-nsl}. This completes the proof of Proposition~\ref{pr-fnl}.
\end{proof}

We are now ready to handle the general case of profinite groups using Corollary~\ref{c-virt} on virtually prosoluble groups and induction on the nonprosoluble length.
First we eliminate infinite Cartesian products of non-abelian finite simple groups.

\begin{lemma}\label{l-cart}
Suppose that $G$ is a profinite group that is a Cartesian product of non-abelian finite simple groups. If every element of $G$ has a countable right Engel sink, then $G$ is finite.
\end{lemma}

\begin{proof}
Suppose the opposite: then $G$ is a Cartesian product of infinitely many non-abelian finite simple groups $G_i$ over an infinite set of indices $i\in I$.

Every non-abelian finite simple groups $S$ contains an element $s\in S$ with a nontrivial smallest right Engel sink $\mathscr R(s)\ne \{1\}$. Actually, any nontrivial element  $s\in S\setminus\{1\}$ has nontrivial minimal right Engel sink. Indeed, otherwise $s$  is a right Engel element of $S$, and right Engel elements of a finite group belong to its hypercentre by Baer's theorem \cite[12.3.7]{rob}.  By Lemma~\ref{l-min-r}, for any $z\in \mathscr R(s)$ we have
\begin{equation*}
z=[s,{}_nx]=[s,{}_{n+m}x]
\end{equation*}
for some $x\in S$ and some $n\geq 1$ and $m\geq 1$, and then also
\begin{equation*}
z=[s,{}_nx]=[s,{}_{n+ml}x]\quad \text{for any}\;\, l\in {\mathbb N}.
\end{equation*}

For every $i$, we choose a nontrivial element $g_i\in G_i$,  a nontrivial element  $z_{i}\in \mathscr R(g_i)\subseteq G_i$, and the corresponding $x_i\in G_i$ such that for some $n_i\geq 1$ and $m_i\geq 1$
\begin{equation}\label{e-cycl2}
z_i=[g_i,{}_{n_i}x_i]=[g_i,{}_{n_i+m_il}x_i]\quad \text{for any}\;\, l\in {\mathbb N}.
\end{equation}

Consider the element
$$
g=\prod _{i\in I} g_{i}.
$$
For any subset $J\subseteq I$, consider the element
$$
x_{J}=\prod _{j\in J} x_{j}.
$$
If  $\mathscr R(g)$ is any right Engel sink of $g$ in $G$, then for some $k\in {\mathbb N}$ the commutator $[g, \,{}_kx_{J}]$ belongs to $\mathscr R(a)$. Because of the properties \eqref{e-cycl2}, all the components of $[g, \,{}_kx_{J}]$ in the factors $G_j$ for $j\in J$ are nontrivial, while all the other components in $G_i$ for $i\not\in J$ are trivial by construction. Therefore for different subsets $J\subseteq I$ we thus obtain different elements of  $\mathscr R(g)$. The infinite set $I$ has at least continuum of different subsets, whence $\mathscr R(g)$ is uncountable, contrary to $g$ having a countable Engel sink by the hypothesis.
\end{proof}

\begin{theorem}\label{t4}
Suppose that $G$ is a profinite group in which every element has a countable Engel sink. Then $G$ has a finite normal subgroup $N$ such that $G/N$ is locally nilpotent.
\end{theorem}

\begin{proof}
By Proposition~\ref{pr-fnl} the group $G$ has finite nonprosoluble length $l$. This means that $G$ has a normal series
\begin{equation*}
1=L_0\leqslant R_0 <  L_1\leqslant R_1< L_1  \leqslant \dots \leqslant R_{l}=G
\end{equation*}
in which each quotient $L_i/R_{i-1}$ is a (nontrivial) Cartesian product of non-abelian finite simple groups, and each quotient $R_i/L_{i}$ is prosoluble (possibly trivial).
We argue by induction on $l$. When $l=0$, the group $G$ is prosoluble, and the result follows by Theorem~\ref{t3}.

Now let $l\geq 1$. By Lemma~\ref{l-cart} each of the nonprosoluble factors $L_i/R_{i-1}$ is finite. In particular, the subgroup $L_1$ is virtually prosoluble, and therefore $\gamma _{\infty}(L_1)$ is finite by Corollary~\ref{c-virt}. The quotient $R_1/ \gamma _{\infty}(L_1)$ is prosoluble by Lemma~\ref{l-prosol-by-prosol}. Hence the  nonprosoluble length of $G/\gamma _{\infty}(L_1)$ is $l-1$. By the induction hypothesis we obtain  that $\gamma _{\infty}(G/\gamma _{\infty}(L_1))$ is finite, and therefore $\gamma _{\infty}(G)$ is finite. By Theorem~\ref{t2} the quotient $G/\gamma _{\infty}(G)$ is locally nilpotent, and the proof is complete.
\end{proof}

\section{Compact groups}
\label{s-comp}
In this section we prove the main Theorem~\ref{t-main}
about compact groups with countable right Engel sinks.  We use the structure theorems for compact groups and the results of the preceding section on profinite groups.  Parts of the proof are similar to the proof of the main results of \cite{khu-shu,khu-shu191} about finite and countable left Engel sinks. In the end we reduce the proof to the situation where every element has a finite left Engel sink and then apply Theorem~\ref{t-left}.

 By the well-known structure theorems (see, for example, \cite[Theorems~9.24 and 9.35]{hof-mor}), the connected component of the identity $G_0$ of a compact (Hausdorff) group $G$ is a divisible normal subgroup such that $G_0/Z(G_0)$ is a Cartesian product of (non-abelian) simple compact Lie groups, while the quotient $G/G_0$ is a profinite group. (Recall that a group $H$ is said to be \emph{divisible} if for every $h\in H$ and every positive integer $k$ there is an element $x\in H$ such that $x^k=h$.)

We shall be using the following lemma from \cite{khu-shu}.

\begin{lemma}[{\cite[Lemma~5.3]{khu-shu}}] \label{l-eng}
Suppose that $G$ is a compact group in which every element has a finite left Engel sink and the connected component of the identity $G_0$ is abelian. Then for every $g\in G$ and for any $x\in G_0$ we have
$$
[x,\,{}_kg]=1\quad \text{for some}\;\,k=k(x,g)\in {\mathbb N}.
$$
\end{lemma}

For compact groups with countable right Engel sinks, we begin with eliminating simple Lie groups.

 \begin{lemma}\label{l-lie}
A non-abelian simple compact Lie group contains an element all of whose right
Engel sinks are uncountable.
\end{lemma}

\begin{proof}
It is well known  that any non-abelian compact Lie group $G$ contains a subgroup isomorphic either to $SO_3 (\mathbb{R})$ or $SU_2( \mathbb{C})$ (see, for example, \cite[Proposition~6.46]{hof-mor}), and therefore in any case, a section isomorphic to $SO_3 (\mathbb{R})$. Since the property that every element has a countable right Engel sink is  inherited by sections, it is sufficient to consider the case $G=SO_3 (\mathbb{R})$.

Consider the following elements of $SO_3 (\mathbb{R})$:
$$
a_\vartheta =\begin{pmatrix} \cos \vartheta &\sin \vartheta &0\\-\sin \vartheta  &\cos \vartheta &0\\0&0&1\end{pmatrix},\quad \vartheta\in  \mathbb{R},
$$
and
$$
g=\begin{pmatrix} -1&0&0\\0&1&0\\0&0&-1\end{pmatrix}.
$$
We have
\begin{align*}
[a_\vartheta ,g]=a_\vartheta ^{-1}a^g&=\begin{pmatrix} \cos (-\vartheta )&\sin (-\vartheta )&0\\-\sin (-\vartheta ) &\cos (-\vartheta )&0\\0&0&1\end{pmatrix}\cdot \begin{pmatrix} \cos \vartheta &-\sin \vartheta &0\\\sin \vartheta  &\cos \vartheta &0\\0&0&1\end{pmatrix}\\ &=\begin{pmatrix} \cos (-2\vartheta )&\sin (-2\vartheta )&0\\-\sin(-2\vartheta ) &\cos (-2\vartheta )&0\\0&0&1\end{pmatrix}=a_{-2\vartheta },
\end{align*}
and then by induction,
$$
[a_\vartheta ,\,{}_ng]=a_{(-2)^n\vartheta }.
$$
Therefore  any left Engel sink of $g$ must contain, for every $\vartheta \in \mathbb{R}$, an element of the form $a_{(-2)^{n(\vartheta )}\vartheta }$ for some $n(\vartheta )\in {\mathbb N}$. Since for $\vartheta $ we can choose continuum elements of $\mathbb{R}$ that are linearly independent over  $\mathbb{Q}$,
any left Engel sink of $g$ must be uncountable.

All the elements $a_{\vartheta}$ form an abelian subgroup $A$, which is normalized by $g$. The above arguments actually show that  any left Engel sink of $g$ in $A\langle g\rangle$ must be uncountable.

Since the  group $A\langle g\rangle$ is metabelian, by Lemma~\ref{l-metab} a right Engel sink of $g=g^{-1}$ in $A\langle g\rangle$ is a left Engel sink of~$g$ in  $A\langle g\rangle$. Therefore any right  Engel sink of $g$ must be uncountable.
\end{proof}

The next lemma is a step towards proving that every element has a finite Engel sink.

\begin{lemma}\label{l-ab}
Suppose that $G$ is a compact group in which every element has a countable right Engel sink. If $G$ has an abelian subgroup $A$ with locally nilpotent quotient $G/A$, then every element of $G$ has a finite left Engel sink.
\end{lemma}

\begin{proof}
Since $G/A$ is locally nilpotent, for showing that an element $g\in G$ has a  finite left  Engel sink we can obviously assume that $G=A\langle g\rangle$.
Since the group $A\langle g\rangle$ is metabelian, by Lemma~\ref{l-metab} every element of it also has a countable left Engel sink. Then $g$ has a finite left Engel sink by Theorem~\ref{t-left}.
  \end{proof}

We are now ready to prove the main result.

\begin{theorem}\label{t5}
Suppose that $G$ is a compact group in which every element has a countable right Engel sink. Then $G$ has a finite normal subgroup $N$ such that $G/N$ is locally nilpotent.
\end{theorem}

\begin{proof} In view of Lemma~\ref{l-lie}, the connected component of the identity $G_0$ is an abelian divisible normal subgroup.

\begin{lemma} \label{l-eng2}
For every $g\in G$ and for any $x\in G_0$ we have
$$
[x,\,{}_kg]=1\quad \text{for some}\;\,k=k(x,g)\in {\mathbb N}.
$$
\end{lemma}

\begin{proof}
We can obviously assume that $G=G_0\langle g\rangle$. The group $G_0\langle g\rangle$ satisfies the hypothesis of Lemma~\ref{l-ab} and therefore every element in it has a finite left Engel sink. Then for any $x\in G_0$ we have $[x,\,{}_kg]=1$ for some $k=k(x,g)\in {\mathbb N}$  by Lemma~\ref{l-eng}.
\end{proof}

We proceed with the proof of Theorem~\ref{t5}. Applying Theorem~\ref{t4} to the profinite group $\bar G=G/G_0$ we obtain a finite normal subgroup $D$ with locally nilpotent quotient. Then every element $g\in\bar G$ has a finite smallest left Engel sink $\bar{\mathscr E}(g)$ contained in $D$. Consider the subgroup generated by all such sinks:
$$
E=\langle \bar{\mathscr E}(g)\mid g\in \bar G\rangle\leq D.
$$
Clearly, $\bar{\mathscr E}(g)^h=\bar{\mathscr E}(g^h)$ for any $h\in\bar G$; hence  $E$ is a normal finite subgroup of $\bar G$. Note that $\bar G/E$ is also locally nilpotent by
Theorem~\ref{t-wz} as an Engel profinite group.

We now consider the action of  $\bar G$ by automorphisms  on $G_0$ induced by conjugation.

\begin{lemma}\label{l-central}
The subgroup $E$ acts trivially on $G_0$.
\end{lemma}

\begin{proof}
In the proof of this lemma, we consider $G_0$ as an abstract abelian divisible group. Thus,  $G_0$ is a direct product $A_0\times\bigoplus _pA_p$ of a torsion-free divisible group $A_0$ and divisible Sylow $p$-subgroups $A_p$ over various primes $p$. Clearly, every Sylow subgroup $A_p$ is normal in $G$.

First we show that $E$ acts trivially on each $A_p$. It is sufficient to show that for every $g\in \bar G$ every element  $z\in \bar{\mathscr E}(g)$ acts trivially on $A_p$. Consider the action of $\langle z, g\rangle$ on $A_p$. Note that $\langle z, g\rangle=\langle z^{\langle  g\rangle}\rangle\langle  g\rangle$, where $\langle z^{\langle  g\rangle}\rangle$ is  a finite $g$-invariant subgroup, since it is contained in the finite subgroup $E$. For any $a\in A_p$ we have $[a,\,{}_kg]=1$ for some $k=k(a,g)\in {\mathbb N}$  by Lemma~\ref{l-eng2}. Hence the subgroup
$$
\langle a^{\langle g\rangle}\rangle=\langle a,[a,g], [a,g,g],\dots \rangle
$$
is a finite $p$-group; note that this subgroup is  $g$-invariant.
 The images of $\langle a^{\langle g\rangle}\rangle$ under the action of elements of the finite group  $\langle z^{\langle  g\rangle}\rangle$  generate a finite $p$-group $B$, which is $\langle z, g\rangle$-invariant. It follows from Lemma~\ref{l-eng2} that $\langle z, g\rangle/C_{\langle z, g\rangle}(B)$
 must be a $p$-group. Indeed, otherwise there is a $p'$-element $y\in \langle z, g\rangle/C_{\langle z, g\rangle}(B)$ that acts non-trivially on the Frattini quotient $V=B/\Phi (B)$. Then $[[V,y],y]=[V,y]\ne 1$ and $C_{[V,y]}(y)=1$, whence  $[V,y]=\{[v,y]\mid v\in [V,y]\}$ and therefore also $[V,y]=\{[v,\,{}_ny]\mid v\in [V,y]\} $ for any $n$, contrary to Lemma~\ref{l-eng2}. Thus, $\langle z, g\rangle/C_{\langle z, g\rangle}(B)$ is a finite $p$-group. But since $z\in \bar{\mathscr E}(g)$, by Lemma~\ref{l-min} we have  $z=[z,{}_mg]$ for some $m\in {\mathbb N}$. Since a finite $p$-group is nilpotent, this implies that $z\in C_{\langle z, g\rangle}(B)$. In particular, $z$ centralizes $a$. Thus, $E$ acts trivially on $A_p$, for every prime~$p$.

 We now show that $E$ also acts trivially on the quotient $W=G_0/\bigoplus _pA_p$  of $G_0$ by its torsion part. Note that $W$ can be regarded as a vector space over ${\mathbb Q}$. Every element $y\in E$ has finite order and therefore by Maschke's theorem   $W=[W,y]\oplus C_W(y)$ and $[W,y]=\{[w,\,{}_ny]\mid w\in [W,y]\} $  for any $n$. If $[W,y]\ne 0$, then this contradicts Lemma~\ref{l-eng2}.

 Thus, $E$ acts trivially both on $W$ and on  $\bigoplus _pA_p$. Then any automorphism $\eta$ of $G_0$ induced by conjugation by $h\in E$ acts on every element $a\in A_0$ as $a^{\eta}=a^h=at$, where $t=t(a,h)$ is an element of finite order in $G_0$. Then  $a^{\eta ^i}=at^i$, and therefore the order of $t$ must divide the order of $\eta$.

 Assuming the action of $E$ on $G_0$ to be non-trivial, choose an element $h\in E$  acting on $G_0$ as an automorphism $\eta$ of some prime order $p$. Then there is $a\in A_0$ such that $a^h=as$, where $s\in A_p$ has order $p$. There is an element $a_1\in A_0$ such that $a_1^{p}=a$. Then $a_1^h=a_1s_1$, where $s_1^{p}=s$. Thus, $|s_1|=p^{2}$, and therefore  $p^{2}$ divides the order of $\eta $. We arrived  at a contradiction with $|\eta |=p$.
\end{proof}

We now finish the proof of Theorem~\ref{t5}. Let $F$ be the full inverse image of $E$ in $G$. Then we have normal subgroups $G_0\leq F\leq G$ such that $G/F$ is locally nilpotent, $F/G_0$ is finite, and $G_0$ is contained in the centre of $F$ by Lemma~\ref{l-central}. Since $F$ has centre of finite index, the derived subgroup $F'$ is finite by Schur's theorem \cite[Satz~IV.2.3]{hup}. The quotient $G/F'$ is an extension of an abelian subgroup by a locally nilpotent group. Hence every element of $G/F'$ has a finite left Engel sink by Lemma~\ref{l-ab}. By Theorem~\ref{t-left} the group  $G/F'$ has a finite normal subgroup with locally nilpotent quotient. The full inverse image of this subgroup is a required finite normal subgroup $N$ such that $G/N$ is locally nilpotent.
The proof of  Theorem~\ref{t5} is complete.
 \end{proof}

 \section*{Acknowledgements}
The authors thank John Wilson for stimulating discussions.

The first author was supported by Mathematical Center in Akademgorodok, the agreement with Ministry of Science and High Education of the Russian Federation no.~075-15-2019-1613. 
The second author was supported by FAPDF and CNPq-Brazil.

\end{document}